\documentclass[12pt]{article}

\usepackage{latexsym,amssymb,colordvi}
\usepackage{graphicx,amsmath,amsthm}
\usepackage[]{todonotes}
\usepackage{subcaption}
\usepackage{tikz-qtree}
\usepackage{url}

\newtheorem{theorem}{Theorem}[section]
\newtheorem{proposition}[theorem]{Proposition}
\newtheorem{lemma}[theorem]{Lemma}
\newtheorem{definition}[theorem]{Definition}

\newtheorem{remark}[theorem]{Remark}
\newtheorem{assumption}[theorem]{Assumption}

\newcommand{\kblue}{\textcolor[rgb]{0.00,0.00,0.00}} %
\usepackage{graphicx}
\usepackage{mathtools}
\usepackage{amssymb,amsmath}
\usepackage{algorithm}
\usepackage{algpseudocode}
\usepackage{enumerate}
\usepackage[inline]{enumitem}
\usepackage{url}
\usepackage{multirow}
\usepackage{placeins}
\usepackage{bigstrut}
\usepackage{rotating}
\usepackage{lscape}
\usepackage{longtable}
\usepackage[]{todonotes}

\def\ASl1{\texttt{AS-$\ell_1$}}
\def\N{\mathbb{N}}
\newcommand{\R}{\mathbb R} 
\DeclarePairedDelimiter{\norm}{\lVert}{\rVert} 
\DeclarePairedDelimiter{\abs}{\lvert}{\rvert} 
\DeclareMathOperator*{\argmin}{Argmin} 
\DeclareMathOperator*{\argmax}{Argmax} 
\DeclareMathOperator{\sign}{sgn} 

\def\eps{\epsilon}

\begin{document}

\title{Minimization over the $\ell_1$-ball using an active-set non-monotone projected gradient}

\author{A.~Cristofari\footnote{Dipartimento di Matematica ``Tullio Levi-Civita'', Universit\`a di Padova
              ({\tt andrea.cristofari@unipd.it, rinaldi@math.unipd.it})}\addtocounter{footnote}{5} \and M. De Santis\footnote{Dipartimento di Ingegneria Informatica, Automatica e Gestionale, Sapienza Universit\`{a} di Roma
              ({\tt mdesantis@diag.uniroma1.it, lucidi@diag.uniroma1.it})} \\
              \and~S. Lucidi$^{**}$~\and~F.~Rinaldi$^*$}
 \date{\today}

\maketitle

\begin{abstract}
The $\ell_1$-ball is a nicely structured feasible set that is widely used in many fields (e.g., machine learning, statistics and signal analysis) 
to enforce some sparsity in the model solutions. In this paper, we devise an active-set strategy for efficiently dealing with minimization problems over 
the $\ell_1$-ball and embed it into a tailored algorithmic scheme that makes use of a non-monotone first-order approach to explore the given subspace at
each iteration. We prove global convergence to stationary points. Finally, we report numerical experiments, on two different classes of instances, showing the effectiveness of the algorithm.

\end{abstract}

\section{Introduction}
In this paper, we focus on the following problem:
\begin{equation}\label{probl1}
\begin{split}
& \min \, \varphi(x) \\
& \norm{x}_1 \le \tau,
\end{split}
\end{equation}
where $\varphi \colon \R^n \to \R$ is a function whose gradient is Lipschitz continuous with constant $L>0$, $\norm x_1$ denotes the $\ell_1$-norm of the vector $x$ and $\tau$ is a suitably chosen positive parameter.

Problem~\eqref{probl1} includes, as a special case, the so called \textit{LASSO problem}, obtained when
\[
\varphi(x) = \norm{Ax-b}^2,
\]
with $A$ and $b$ being a $m \times n$ matrix and a $m$-dimensional vector, respectively. Here and in the following, $\|\cdot\|$ denotes the Euclidean norm.
Loosely speaking, in LASSO problems the $\ell_1$-norm constraint is able to induce sparsity in the final solution,
and then these problems are widely used in statistics to build regression models with a small number of non-zero coefficients~\cite{efron2004least,tibshirani1996regression}.

Standard optimization algorithms (like, e.g., interior-point methods), besides being very expensive when the number of variables increases, do no properly exploit the main features and structure of the considered problem. This is the reason why, in the last decade, a number of first-order methods have been considered in the literature to deal with problem \eqref{probl1}. Those methods can be divided into two main classes: projection-based approaches, like, e.g.,  gradient-projection methods \cite{duchi2012projected,schmidt2008structure} and limited-memory projected quasi-Newton methods \cite{schmidt2009optimizing}, which efficiently handle the problem by making use of tailored projection strategies \cite{condat2016fast,duchi2008efficient}, and projection-free methods, like, e.g., Frank-Wolfe variants \cite{bomze2019first,bomze:2019,jaggi2013revisiting,lacoste:2015}, that embed a cheap linear minimization oracle. 

As highlighted before, the main goal when using the $\ell_1$ ball is to get very sparse solutions (i.e., solutions with many zero components). In this context, it hence makes sense to devise strategies that allow to quickly identify the set of zero components in the optimal solution. This would indeed  guarantee a significant speed-up of the optimization process. A number of active\kblue{-}set strategies for structured feasible sets is available in the literature
(see, e.g.,~\cite{bertsekas:1982,birgin2002large,bras:2017,cristofari:2017,cristofari:2020,diserafino:2017,facchinei1998accurate,facchinei1998active,hager:2006,hager:2016a,hager:2016b,more1989algorithms} and references therein), but none of those directly handles the $\ell_1$ ball.

In this paper, inspired by the work carried out in \cite{cristofari:2020}, we propose a tailored active\kblue{-}set strategy for problem \eqref{probl1} and embed it into a  first-order projection-based algorithm. 
At each iteration, the method first sets to zero the variables that are guessed to be zero at the final solution.
This is done by means of the tailored active-set estimate, which aims at identifying the manifold where the solutions of problem \eqref{probl1} lie, while guaranteeing, thanks to a descent property, a reduction of the objective function at each iteration.
Then, the remaining variables, i.e., those variables estimated to be non-zero at the final solution, are suitably modified by means of a non-monotone gradient-projection step.

The paper is organized as follows. In Section \ref{sec:asest}, we describe 
the active-set strategy and analyze the descent property connected to it. We then devise, in Section \ref{sec:alg}, our first-order optimization algorithm and carry out a global convergence analysis. We further report a numerical comparison with some well-known first order methods using two different classes of  $\ell_1$-constrained problems (that is, LASSO and constrained sparse logistic regression) in Section~\ref{sec:results}. Finally, we draw some conclusions in Section~\ref{sec:conc}.

\section{The active-set estimate}\label{sec:asest}
Since the feasible set of problem~\eqref{probl1} 
is convex and can be written as convex combination of the 
vectors $\pm \tau e_i$, $i=1,\ldots,n$, 
we can characterize the stationary points as follows.
\begin{definition}\label{def:stat}
A feasible point $x^*$ of problem~\eqref{probl1} is stationary if and only if
\begin{equation}\label{eq:stat}
\begin{split}
& \nabla\varphi(x^*)^T(\tau e_i - x^*) \geq 0, \quad i = 1,\ldots, n, \\
& \nabla\varphi(x^*)^T(-\tau e_i - x^*) \geq 0, \quad i = 1,\ldots, n.
\end{split}
\end{equation}
\end{definition}

\kblue{In the next proposition, we state some ``complementarity-type'' conditions for stationary points of problem~\eqref{probl1}.
\begin{proposition}\label{prop-staz-l1}
Let $x^*$ be a stationary point of problem~\eqref{probl1}. Then
\begin{enumerate}[label=(\roman*), leftmargin=*]
\item $x_i^*>0 \, \Rightarrow \, \nabla \varphi(x^*)^T (\tau e_i - x^*)=0$,
\item $x_i^*<0 \, \Rightarrow \, \nabla\varphi(x^*)^T (-\tau e_i - x^*)=0$.
\end{enumerate}		
\end{proposition}
\begin{proof}
If $|x_i^*| = \tau$, then $x^* = \tau \, \sign(x^*_i) \, e_i$
and the result trivially holds.
To prove point~(i), now let $0 < x_i^* <\tau$.
Taking into account (\ref{eq:stat}), by contradiction we assume that
\begin{equation}\label{compl1}
\nabla\varphi(x^*)^T (\tau e_i - x^*)>0.
\end{equation}
Let $d^+\in \R^n$ be defined as follows:
\[
d^+=\frac{x_i^*}{\tau-x_i^*}(x^*-\tau e_i).
\]
We have
\begin{equation}\label{compl2}
\begin{split}
\|x^*+d^+\|_1 & =\biggr(1+ \frac{x_i^*}{\tau-x_i^*}\biggl)\sum_{j\neq i}  |x^*_j|+\biggr|x_i^*+ \frac{x_i^*}{\tau-x_i^*}(x_i^*-\tau)\biggl| \\
& \le\biggr(1+ \frac{x_i^*}{\tau-x_i^*}\biggl)(\tau -|x_i^*|)=\tau,
\end{split}
\end{equation}
so that $d^+$ is a feasible direction in $x^*$. 
Therefore,~\eqref{compl1} and~\eqref{compl2} imply that $d^+$ is a feasible descent direction for $\varphi(\cdot)$ in $x^*$. This contradicts the fact that $x^*$ is a stationary point of problem~\eqref{probl1} and point (i) is proved. To prove point~(ii), we can use the same arguments as above,
considering $-\tau < x_i^* < 0$ and,
assuming by contradiction that $\nabla\varphi(x^*)^T (-\tau e_i - x^*)>0$,
we obtain that
\[
d^-=\frac{|x_i^*|}{\tau+x_i^*}(x^*+\tau e_i)
\]
is such that $\|x^*+d^-\|_1 \le \tau$, that is,
$d^-$ is a feasible and descent direction for $\varphi(\cdot)$ in $x^*$,
leading to a contradiction.
\end{proof}
}

With little abuse of standard terminology, given a stationary point $x^*$ we say that a variable $x^*_i$ is \textit{active} if $x^*_i=0$,
whereas a variable $x^*_i$ is said to be \textit{non-active} if $x^*_i \ne 0$.
We can thus define the \textit{active set} $\bar A_{\ell_1}(x^*)$ and the \textit{non-active set} $\bar N_{\ell_1}(x^*)$ as follows:
\[
\bar A_{\ell_1}(x^*) = \{i \colon x^*_i=0\}, \quad \bar N_{\ell_1}(x^*) = \{1,\ldots,n\} \setminus \bar A_{\ell_1}(x^*).
\]

Now, we show how we estimate these sets starting from any feasible point $x$ of problem \eqref{probl1}.
%
%
In order to obtain such an estimate we first need to suitably reformulate our problem~\eqref{probl1} by introducing a dummy variable $z$.
Let $\bar \varphi(x,z) \colon \R^{n+1} \to \R$ be the function defined as $\bar \varphi(x,z) = \varphi(x)$ for all $(x,z)$.
Problem~\eqref{probl1} can then be rewritten as
\begin{equation}\label{probl1_z}
\begin{split}
& \min \, \bar \varphi(x,z) \\
& \|x\|_1 +z\leq \tau,\\
& z \geq 0.
\end{split}
\end{equation}

Every feasible point of problem~\eqref{probl1_z} can be expressed as convex combination of
$\{\pm \tau e_1,\ldots, \pm \tau e_n, \tau e_{n+1}\} \subset \R^{n+1}$.
Therefore, we can define the following matrix, where $I$ denotes the $n \times n$ identity matrix:
\[
\bar M = \tau
\left[\scriptsize
\begin{array}{c | c | c}
\mbox{\large{$I$}} & \mbox{\large{$-I$}} & \begin{matrix}0\\ \vdots\\ 0\end{matrix} \\
\hline
\begin{matrix}0 &\ldots & 0\end{matrix} & \begin{matrix}0 &\ldots & 0\end{matrix} &1
\end{array}
\right] \in \R^{(n+1) \times (2n+1)},
\]
and we obtain the following reformulation of~\eqref{probl1}
as a minimization problem over the unit simplex:
\begin{equation}\label{probl1_eq}
\begin{split}
& \min \, f(y) = \bar \varphi(\bar M y) \\
& e^T y  = 1, \\
& y \ge 0.
\end{split}
\end{equation}
Note that, given any feasible point $x$ of problem~\eqref{probl1}, we can compute a feasible point $y$ of problem~\eqref{probl1_eq} such that
\begin{equation}\label{y_x}
\begin{split}
& y_i = \frac 1 {\tau} \max\{0,x_i\}, \quad i = 1,\dots,n, \\
& y_{n+i} = \frac 1 {\tau} \max\{0,-x_i\}, \quad i = 1,\dots,n, \\
& y_{2n+1} = \frac {\tau-\|x\|_1} {\tau}.
\end{split}
\end{equation}

The rationale behind our approach is sketched in the three following points:
\begin{enumerate}[label=(\roman*), leftmargin=*]
%
\item    
    For any feasible point $x$ of problem~\eqref{probl1},
    by~\eqref{y_x} we can compute a feasible point $y$ of problem~\eqref{probl1_eq} such that
    \begin{equation}\label{y0}
    y_i = 0 \, \Leftrightarrow \, x_i \le 0 \quad \text{and} \quad y_{n+i} = 0 \, \Leftrightarrow \, x_i \ge 0, \qquad i = 1,\ldots,n.
    \end{equation}

\item According to~\eqref{y0}, for every feasible point $x$ of problem~\eqref{probl1} we have that
    \begin{equation}\label{x0}
    x_i = 0 \, \Leftrightarrow \, y_i = y_{n+i} = 0, \quad i = 1,\ldots,n.
    \end{equation}
    Thus, it is natural to estimate a variable $x_i$ as active at $x^*$
    if both $y_i$ and $y_{n+i}$ are estimated to be zero at the point corresponding to $x^*$ in the $y$ space.
    To estimate the zero variables among $y_1,\ldots,y_{2n+1}$ we use the
    active-set estimate described in~\cite{cristofari:2020}, specifically devised for minimization problems over the unit simplex.
\item Then, we are able to go back in the original $x$ space to obtain an active-set estimate of problem~\eqref{probl1} without explicitly considering the
     variables $y_1,\ldots,y_{2n+1}$ of the reformulated problem.
\end{enumerate}

\begin{remark}
 The introduction of the dummy variable $z$ is needed in order to get a reformulation of problem~\eqref{probl1} \kblue{as a minimization problem over the unit simplex} satisfying~\eqref{y0}.
 Since every feasible point $x$ of problem~\eqref{probl1} can be expressed as \kblue{a} convex combination of the vertices of the polyhedron $\{x\in \R^n:\; \|x\|_1 \leq \kblue{\tau}\}$,
 a straightforward reformulation of problem~\eqref{probl1} would \kblue{then} be the following:
 \begin{equation}\label{ref_st}
 \min \{\varphi(My) \colon e^T y = 1, \,  y \ge 0\},
 \end{equation}
 with $M = \tau \,\begin{bmatrix} & I & \vline & \Large{-I} & \end{bmatrix}$.
 However, this reformulation does not work for our purposes, as there exist feasible points $x$ of problem~\eqref{probl1} for which no $y$ feasible for problem \eqref{ref_st}  satisfying~\eqref{y0} can be found.
 In particular, if $x$ is in the interior of the $\ell_1$-ball (e.g., the origin), we cannot find any $y$ feasible for problem \eqref{ref_st} such that~\eqref{y0} holds.
\end{remark}

Considering problem~\eqref{probl1_eq}
 and using the active-set estimate proposed in~\cite{cristofari:2020} for minimization problems over the unit simplex,
given any feasible point $y$ of problem~\eqref{probl1_eq} we define:
\begin{gather}
A(y) = \{i \colon y_i \le \epsilon \nabla f(y)^T(e_i - y)\}, \label{as1} \\
N(y) = \{i \colon y_i > \epsilon \nabla f(y)^T(e_i - y)\},  \label{nas1}
\end{gather}
where $\epsilon$ is a positive parameter.
$A(y)$ contains the indices of the variables that are estimated to be zero at a certain stationary point
and $N(y)$ contains the indices of the variables that are estimated to be positive at the same stationary point
(see~\cite{cristofari:2020} for details of how these formulas are obtained).
As mentioned above, taking into account~\eqref{x0}, we estimate a variable $x_i$ as active for problem~\eqref{probl1}
if both $y_i$ and $y_{n+1}$ are estimated to be zero.
Namely,
\begin{subequations}
\begin{gather}
A_{\ell_1}(x) = \bigl\{i \in \{1,\ldots,n\} \colon i \in A(y) \, \text{ and } \, (n+i) \in A(y) \bigr\}, \\
N_{\ell_1}(x) = \bigl\{i \in \{1,\ldots,n\} \colon i \in N(y) \, \text{ or } \, (n+i) \in N(y) \bigr\}. \label{N_l1}
\end{gather}
\end{subequations}

Now we show how $A_{\ell_1}(x)$ and $N_{\ell_1}(x)$ can be expressed without explicitly considering the
variables $y$ and the objective function $f(y)$ of the reformulated problem.
This allows us to work in the original $x$ space, avoiding to double the number of variables in practice.

To obtain the desired relations, first observe that
\begin{equation}\label{grad_f}
\nabla f(y) = \bar M^T \nabla \bar \varphi (x)
= \tau \begin{bmatrix} \nabla \varphi(x) \\ -\nabla \varphi(x) \\ 0 \end{bmatrix}^T,
\end{equation}
and
\[
\nabla f(y)^T y = \nabla \bar \varphi(x)^T \bar M y =
\begin{bmatrix}
\nabla \varphi(x)^T & 0
\end{bmatrix} \bar M y = \nabla \varphi(x)^T x.
\]
Let us distinguish two cases:
\begin{enumerate}[label=(\roman*), leftmargin=*]
\item $x_i \ge 0$.
    Recalling~\eqref{as1}--\eqref{nas1}, we have that $i \in A(y)$ if and only if
    \begin{equation}\label{y_act_x_pos_1}
    \begin{split}
     0 \le \frac 1 \tau x_i =y_i & \le \epsilon \nabla f(y)^T(e_i - y) = \epsilon (\nabla_i f(y) - \nabla f(y)^T y) \\
                                    & = \epsilon (\tau \nabla_i \varphi(x) - \nabla \varphi(x)^Tx) = \epsilon \nabla \varphi(x)^T (\tau e_i - x)
    \end{split}
    \end{equation}
    and $(n+i) \in A(y)$ if and only if
    \begin{equation}\label{y_act_x_pos_2}
    \begin{split}
     -\frac 1 {\tau} x_i \le 0 = y_{n+i} & \le \epsilon \nabla f(y)^T(e_{n+i} - y) = \epsilon (\nabla_{n+i} f(y) - \nabla f(y)^T y)\\
                                           & = \epsilon (-\tau \nabla_i \varphi(x) - \nabla \varphi(x)^Tx) = -\epsilon \nabla \varphi(x)^T (\tau e_i + x).
    \end{split}
    \end{equation}
\item $x_i < 0$.
    Similarly to the previous case, we have that $i \in A(y)$ if and only if
    \begin{equation}\label{y_act_x_neg_1}
    \begin{split}
    \frac 1 {\tau} x_i < 0 = y_i & \le \epsilon \nabla f(y)^T(e_i - y) = \epsilon (\nabla_i f(y) - \nabla f(y)^T y) \\
                                   & = \epsilon (\tau \nabla_i \varphi(x) - \nabla \varphi(x)^Tx) = \epsilon \nabla \varphi(x)^T (\tau e_i - x)
    \end{split}
    \end{equation}
    and $(n+i) \in A(y)$ if and only if
    \begin{equation}\label{y_act_x_neg_2}
    \begin{split}
    0 < -\frac 1 \tau x_i = y_{n+i} & \le \epsilon \nabla f(y)^T(e_{n+i} - y) = \epsilon (\nabla_{n+i} f(y) - \nabla f(y)^T y)\\
                                      & = \epsilon (-\tau \nabla_i \varphi(x) - \nabla \varphi(x)^Tx) = -\epsilon \nabla \varphi(x)^T (\tau e_i + x).
    \end{split}
    \end{equation}
\end{enumerate}
From~\eqref{y_act_x_pos_1}, \eqref{y_act_x_pos_2}, \eqref{y_act_x_neg_1} and~\eqref{y_act_x_neg_2},
we thus obtain
\begin{gather}
 \begin{split}
A_{\ell_1}(x) = \{i \colon &\epsilon\, \tau \nabla \varphi(x)^T (\tau e_i + x) \le 0 \le  x_i \le \epsilon\, \tau \nabla \varphi(x)^T (\tau e_i - x)\ \mbox{or} \\
  &\epsilon\, \tau \nabla \varphi(x)^T (\tau e_i + x) \le x_i \le 0 \le \epsilon\, \tau \nabla \varphi(x)^T (\tau e_i - x)\}, \label{act_set_l1}
\end{split} \\
N_{\ell_1}(x) = \{1,\ldots,n\} \setminus A_{\ell_1}(x). \label{non_act_set_l1}
\end{gather}
Let us highlight again that $A_{\ell_1}(x)$ and $N_{\ell_1}(x)$ do not depend on the
variables $y$ and on the objective function $f(y)$ of the reformulated problem,
so no variable transformation is needed in practice to estimate the active set of problem~\eqref{probl1}.
\kblue{
In the following, we prove that under specific assumptions, $\bar A_{\ell_1}(x^*)$ is detected by our active-set estimate, when evaluated in points sufficiently close to a stationary point $x^*$.
\begin{proposition}\label{prop0_Stime-l1}
If $x^*$ is a stationary point of problem~\eqref{probl1}, then there exists
an open ball $\mathcal B(x^*,\rho)$ with center $x^*$ and radius $\rho > 0$
such that, for all $x \in \mathcal B (x^*,\rho)$, we have
\begin{gather}
A_{\ell_1}(x) \subseteq \bar A_{\ell_1} (x^*), \label{stima-loc1} \\
\bar N_{\ell_1}(x^*) \subseteq N_{\ell_1} (x). \label{stima-loc2}
\end{gather}
Furthermore, if the following ``strict-complementarity-type'' assumption holds:
\begin{equation}\label{stima-ass}
x_i^*=0 \, \Rightarrow
\nabla \varphi(x^*)^T (\tau e_i - x^*)>0 \, \land \, \nabla \varphi(x^*)^T (\tau e_i + x^*)<0, 
\end{equation}
then, for all $x \in \mathcal B (x^*,\rho)$, we have
\begin{gather}
A_{\ell_1}(x) = \bar A_{\ell_1} (x^*), \label{stima-loc3} \\
\bar N_{\ell_1}(x^*) = N_{\ell_1} (x). \label{stima-loc4}
\end{gather}
\end{proposition}
\begin{proof}
Let $i\in N_{\ell_1} (x^*) $, then $|x_i^*|>0$.
Proposition~\ref{prop-staz-l1} implies that either 
\[
\nabla\varphi(x^*)^T (\tau e_i - x^*)=0 \quad \text{if} \quad x_i^*>0,
\]
or 
\[
\nabla\varphi(x^*)^T (-\tau e_i - x^*)=0 \quad \text{if} \quad x_i^*<0.
\]
Then, the continuity of $\nabla \varphi$ and the definition of 
$N_{\ell_1} (x)$ imply that there exists 
an open ball $\mathcal B(x^*,\rho)$ with center $x^*$ and radius $\rho > 0$ 
such that, for all $x\in \mathcal B(x^*,\rho)$, we have that $i \in N_{\ell_1} (x)$. This proves~\eqref{stima-loc2} and, consequently, also~\eqref{stima-loc1}.
If~\eqref{stima-ass} holds, the definition of $ N_{\ell_1} (x)$ and the continuity of $\nabla \varphi$ ensures that
$\bar A_{\ell_1} (x^*) \subseteq A_{\ell_1}(x)$ for all $x\in \mathcal B(x^*,\rho)$, implying that~\eqref{stima-loc3} and~\eqref{stima-loc4} hold.
\end{proof}
}


\subsection{Descent property}
So far, we have obtained the active and non-active set estimates~\eqref{act_set_l1}--\eqref{non_act_set_l1}
passing through a variable transformation which allowed us to adapt the active and non\kblue{-}active set estimates proposed
in~\cite{cristofari:2020} to our problem~\eqref{probl1}.

In~\cite{cristofari:2020}, the active and non-active set estimates, designed for minimization problems over the unit simplex, guarantee a decrease in the objective function when setting (some of) the estimated active variables to zero and
moving a suitable estimated non-active variable (in order to maintain feasibility).

In the following, we show that the same property holds for problem~\eqref{probl1} using the active and non-active set
estimates~\eqref{act_set_l1}--\eqref{non_act_set_l1}.
To this aim, in the next proposition we first introduce the index set $J_{\ell_1}(x)$ and \kblue{relate} it with $N_{\ell_1}(x)$.
\begin{proposition}\label{prop0_l1}
Let $x \in \R^n$ be a feasible non-stationary point of problem~\eqref{probl1} and define
\[
J_{\ell_1}(x) = \Bigl\{j \colon j \in \argmax_{i=1,\ldots,n}\, \bigl\{|\nabla_i \varphi(x)|\bigr\} \Bigr\}.
\]
Then,
$J_{\ell_1}(x) \subseteq N_{\ell_1}(x)$.
\end{proposition}

\begin{proof}
Let $y$ be the point given by~\eqref{y_x} and consider the reformulated problem~\eqref{probl1_eq}.
Let $A(y)$ and $N(y)$ be the index sets given in~\eqref{as1}--\eqref{nas1},
that is, the active and non-active set estimates for problem~\eqref{probl1_eq}, respectively.

From the expression of $\nabla f(y)$ given in~\eqref{grad_f}, and exploiting the hypothesis that $x$ is non-stationary
(implying that $\nabla \varphi(x) \ne 0)$, it follows that
\begin{equation}\label{min_g_f}
\min_{i=1,\ldots,2n+1} \{\nabla_i f(y)\} < 0.
\end{equation}
Since $\nabla_{2n+1} f(y) = 0$ (again from~\eqref{grad_f}), it follows that
\[
(2n+1) \notin \argmin_{i=1,\ldots,2n+1} \{\nabla_i f(y)\}.
\]
From Proposition~1 in~\cite{cristofari:2020}, there exists $\nu \in \{1,\ldots,2n\}$ such that
\begin{gather}
\nu \in \displaystyle \argmin_{i=1,\ldots,2n} \{\nabla_i f(y)\}, \label{nu_l1_argmin} \\
\nu \in N(y). \label{nu_l1_in_N}
\end{gather}
In particular, we can rewrite~\eqref{nu_l1_argmin} as
\[
\nabla_{\nu} f(y) = \tau \min_{i=1,\ldots,n} \{\nabla_1 \varphi(x), \ldots, \nabla_n \varphi(x), -\nabla_1 \varphi(x), \ldots, -\nabla_n \varphi(x)\}.
\]
Taking into account~\eqref{min_g_f}, we obtain
\begin{equation}\label{nabla_nu}
-|\nabla_{\nu} f(y)| \le -\tau \abs{\nabla_i \varphi(x)}, \quad \forall i = 1, \ldots,n.
\end{equation}
Now, let $j \in \{1,\ldots,n\}$ be the following index:
\begin{equation}\label{nu_l1}
j =
\begin{cases}
\nu, \quad & \text{ if } \nu \in \{1,\ldots,n\}, \\
\nu-n, \quad & \text{ if } \nu \in \{n+1,\ldots,2n\}.
\end{cases}
\end{equation}
Using again~\eqref{grad_f}, we get $\abs{\nabla_{\nu} f(y)} = \abs{\nabla_j f(y)} = \tau \abs{{\nabla_j \varphi(x)}}$.
This, combined with~\eqref{nabla_nu}, implies that
\[
j \in \argmax_{i=1,\ldots,n}\, \bigl\{\abs{\nabla_i \varphi(x)}\bigr\}.
\]
Finally, using~\eqref{nu_l1_in_N} and~\eqref{nu_l1}, it follows that at least one index between $j$ and $(n+j)$ belongs to $N(y)$.
Therefore, from~\eqref{N_l1} we have that $j \in N_{\ell_1}(x)$ and the assertion is proved. 
\end{proof}

Now, we need an assumption on the parameter $\epsilon$ appearing in~\eqref{act_set_l1}--\eqref{non_act_set_l1}.
It will allow us to prove the subsequent proposition, stating
that $\varphi(x)$ decreases if we set the variables in $A_{\ell_1}(x)$ to zero and suitably move a variable in $J_{\ell_1}(x)$.

\begin{assumption}\label{ass:eps_l1}
Assume that the parameter $\eps$ appearing in the estimates~\eqref{act_set_l1}--\eqref{non_act_set_l1} satisfies the following conditions:
\[
0 < \epsilon \le \frac 1 {\tau^2 n L (2C+1)},
\]
where $C>0$ is a given constant.
\end{assumption}

\begin{proposition}\label{prop1_l1}
Let Assumption~\ref{ass:eps_l1} hold.
Given a feasible non-stationary point $x$ of problem~\eqref{probl1}, let $j \in J_{\ell_1}(x)$ and $I = \{1,\ldots,n\} \setminus \{j\}$.
Let $\hat A_{\ell_1}(x)$ be a set of indices such that
$\hat A_{\ell_1}(x) \subseteq A_{\ell_1}(x)$.
Let $\tilde x$ be the feasible point defined as follows:
\[
\tilde x_{\hat A_{\ell_1}(x)} = 0; \quad \tilde x_{I\setminus \hat A_{\ell_1}(x)} = x_{I\setminus \hat A_{\ell_1}(x)};
\quad \tilde x_j = x_j -\sign(\nabla_j \varphi(x))\displaystyle{\sum_{h \in \hat A_{\ell_1}(x)} \abs{x_h}}.
\]
Then,
\[
\varphi(\tilde x)-\varphi(x)\le - CL \norm{\tilde x-x}^2,
\]
where $C > 0$ is the constant appearing in Assumption~\ref{ass:eps_l1}.
\end{proposition}

\begin{proof}
Define
\begin{equation}\label{eq:hatAp}
\hat A^+ = \hat A_{\ell_1}(x) \cap \{i \colon x_i \ne 0\}.
\end{equation}
Since $\nabla \varphi$ is Lipschitz continuous with constant $L$, from known results (see, e.g.,~\cite{nesterov:2013})
we can write
\[
\begin{split}
\varphi(\tilde x) & \le \varphi(x) + \nabla \varphi(x)^T (\tilde x-x) + \frac L2 \norm{\tilde x-x}^2 \\
                  & = \varphi(x) + \nabla \varphi(x)^T (\tilde x-x) + \frac{L(2C+1)}2 \norm{\tilde x-x}^2 - CL\norm{\tilde x-x}^2
\end{split}
\]
and then, in order to prove the proposition, what we have to show is that
\begin{equation}\label{eqThesis_l1}
\nabla \varphi(x)^T (\tilde x-x) + \frac{L(2C+1)}2 \norm{\tilde x-x}^2 \le 0.
\end{equation}
From the definition of $\tilde x$, we have that
\begin{equation}\label{normdiffl1}
\begin{split}
\norm{\tilde x-x}^2 & = \sum_{i\in \hat A^+} x_i^2 + \Bigg(\sum\limits_{i\in \hat A^+} \kblue{|x_i|} \Bigg)^2 \le \sum_{i\in \hat A^+} x_i^2 + \abs{\hat A^+}\sum_{i\in \hat A^+} x_i^2 \\
& = (\abs{\hat A^+}+1) \sum_{i\in \hat A^+} x_i^2.
\end{split}
\end{equation}
Furthermore,
\begin{equation}\label{ineqGrad_l1}
\begin{split}
\nabla \varphi(x)^T (\tilde x-x) & = -\sum_{i \in \hat A^+} \nabla_i \varphi(x) x_i - \abs{\nabla_j \varphi(x)} \sum_{i \in \hat A^+} |x_i| \\
& = \sum_{i \in \hat A^+} |x_i| (-\nabla_i \varphi(x) \, \sign(x_i) - \abs{\nabla_j \varphi(x)}).
\end{split}
\end{equation}
Since $j \in J_{\ell_1}(x)$, from the definition of $J_{\ell_1}(x)$
it follows that $-\abs{\nabla_i \varphi(x)} \ge -\abs{\nabla_j \varphi(x)}$ for all $i \in \{1,\ldots,n\}$.
Therefore, we can write
\begin{equation}\label{ineqSimp_l1}
\begin{split}
\nabla \varphi(x)^T x & =  \sum_{i=1}^n \nabla_i \varphi(x) \sign(x_i)\, |x_i| \ge \sum_{i=1}^n -\abs{\nabla_j \varphi(x)} \,|x_i| \\
 & = -\abs{\nabla_j \varphi(x)} \,\|x\|_1 \ge -\abs{\nabla_j \varphi(x)}\, \tau.
\end{split}
\end{equation}
Using~\eqref{act_set_l1} and~\eqref{ineqSimp_l1}, for all $i\in \hat A^+$ we have that
\begin{align*}
 x_i & \le  \epsilon \tau (\nabla_i \varphi(x) \tau - \nabla \varphi(x)^T x) \le \epsilon \tau^2  (\nabla_i \varphi(x) + \abs{\nabla_j \varphi(x)}), \\
-x_i & \le -\epsilon \tau (\nabla_i \varphi(x) \tau + \nabla \varphi(x)^T x) \le \epsilon \tau^2 (-\nabla_i \varphi(x) + \abs{\nabla_j \varphi(x)}),
\end{align*}
and then,
\[
\abs{x_i} = \sign(x_i) \, x_i \le \epsilon \tau^2 (\nabla_i \varphi(x) \, \sign(x_i) + \abs{\nabla_j \varphi(x)}), \quad \forall i\in \hat A^+.
\]
Combining this inequality with~\eqref{normdiffl1}, we obtain
\begin{equation}\label{normdiff2_l1}
\norm{\tilde x-x}^2 \le \epsilon \tau^2 (\abs{\hat A^+}+1) \sum_{i \in \hat A^+} |x_i| (\nabla_i \varphi(x) \, \sign(x_i) + \abs{\nabla_j \varphi(x)})
\end{equation}
From~\eqref{ineqGrad_l1} and~\eqref{normdiff2_l1}, it follows that the left-hand side term of~\eqref{eqThesis_l1} is less than or equal to
\[
\biggl(\epsilon \tau^2 \frac{L(2C+1)}2 (\abs{\hat A^+}+1) - 1\biggr) \sum_{i \in \hat A^+} |x_i| (\nabla_i \varphi(x) \, \sign(x_i) + \abs{\nabla_j \varphi(x)})
\]
The desired result is hence obtained, since inequality~\eqref{eqThesis_l1} follows from the assumption we made on~$\epsilon$, using the fact that
$\abs{\hat A^+} \le n - 1$ (as a consequence of Proposition~\ref{prop0_l1}) and
$\sum_{i \in \hat A^+} |x_i| (\nabla_i \varphi(x) \, \sign(x_i) + \abs{\nabla_j \varphi(x)}) \ge 0$ (as a consequence of~\eqref{normdiff2_l1}). 
\end{proof}

We would like to highlight that the parameter $\epsilon$ depends on $n$ by Assumption~\ref{ass:eps_l1}.
However, from the proof of the above proposition, it is clear that $n$ could be replaced by $|\hat A^+|+1$, with
$\hat A^+$ defined as in~\eqref{eq:hatAp}. Note that $|\hat A^+|$ might be much smaller than $n$.

\section{The algorithm}\label{sec:alg}
Based on the active and non-active set estimates described above, we design a suitable active-set algorithm for solving problem~\eqref{probl1},
exploiting the property of our estimates and using an appropriate projected-gradient direction.
At the beginning of each iteration $k$, we have a feasible point $x^k$
and we compute $A_{\ell_1}(x^k)$ and $N_{\ell_1}(x^k)$, which, for ease of notation,
we will refer to as $A_{\ell_1}^k$ and $N_{\ell_1}^k$, respectively.
Then, we perform two main steps:
\begin{itemize} 
\item first, we produce the point $\tilde x^k$ as explained in Proposition~\ref{prop1_l1}, obtaining a decrease in the objective function
    (if $x^k \ne \tilde x^k$);
\item afterward, we move all the variables in $N^k_{\ell_1}$ by computing a projected-gradient direction $d^k$ over the given non-active manifold
    and using a non-monotone Armijo line search. \kblue{In particular, the reference value $\bar\varphi$ for the line search 
    is defined as the maximum among the last $n_m$ function evaluations, with $n_m$ being a positive parameter.}
\end{itemize}

In Algorithm~\ref{alg:ASl1}, we report the scheme of the proposed algorithm, named \texttt{Active-Set algorithm for minimization over the $\ell_1$-ball (\ASl1)}.

\begin{algorithm}[htp]
\caption{\texttt{Active-Set algorithm for minimization over the $\ell_1$-ball (\ASl1)}}
\label{alg:ASl1}
\begin{algorithmic}
{\small
\par\vspace*{0.1cm}
\item[]$\,\,\,1$\hspace*{0.1truecm}\vspace*{0.05cm} Choose a feasible point $x^0$ \kblue{and choose $\epsilon>0$}
\item[]$\,\,\,2$\hspace*{0.1truecm}\vspace*{0.05cm} For $k=0,1,\ldots$
\item[]$\,\,\,3$\hspace*{0.9truecm}\vspace*{0.05cm} If $x^k$ is a stationary point, then STOP
\item[]$\,\,\,4$\hspace*{0.9truecm}\vspace*{0.05cm} Compute $A_{\ell_1}^k = A_{\ell_1}(x^k)$ and $N_{\ell_1}^k = N_{\ell_1}(x^k)$
\item[]$\,\,\,5$\hspace*{0.9truecm}\vspace*{0.05cm} Compute $J_{\ell_1}^k = J_{\ell_1}(x^k)$, choose $j \in J^k_{\ell_1}$ and define $\tilde N_{\ell_1}^k = N_{\ell_1}^k\setminus\{j\}$
\item[]$\,\,\,6$\hspace*{0.9truecm}\vspace*{0.05cm} Set $\tilde x^k_{A_{\ell_1}^k} = 0\,$,
                                  $\,\tilde x^k_{\tilde N_{\ell_1}^k} = x^k_{\tilde N_{\ell_1}^k}\,$ and
                                  $\displaystyle{\,\tilde x^k_j = x^k_j - \sign(\nabla_j \varphi(x^k))\sum_{h \in A_{\ell_1}^k} \abs{x^k_h}}$
\item[]$\,\,\,7$\hspace*{0.9truecm}\vspace*{0.05cm} Compute a projected-gradient type direction $d^k$ such that $d^k_{A_{\ell_1}^k} = 0$
\item[]$\,\,\,8$\hspace*{0.9truecm}\vspace*{0.05cm} Compute a stepsize $\alpha^k\in [0,1]$ by Algorithm~\ref{alg:ls_as}

\item[]$\,\,\,9$\hspace*{0.9truecm}\vspace*{0.05cm} Set $x^{k+1}=\tilde x^k + \alpha^k d^k$
\item[]$10$\hspace*{0.1truecm}\vspace*{0.05cm} End for
\par\vspace*{0.1cm}
}
\end{algorithmic}
\end{algorithm}

\begin{algorithm}[h!]
\caption{\texttt{Non-monotone Armijo line search }}
\label{alg:ls_as}
\begin{algorithmic}
{\small
\par\vspace*{0.1cm}
\item[]$\,\,\,0$\hspace*{0.1truecm}\vspace*{0.05cm} Choose $\delta \in (0,1)$, $n_m>0$  and $\gamma\in (0,1)$
\item[]\,\,\,$1$\hspace*{0.1truecm}\vspace*{0.05cm} Update $\displaystyle{\bar \varphi^k = \max_{0\le i\le \min\{n_m,k\}} \varphi(\tilde x^{k-i})}$
\item[]$\,\,\,2$\hspace*{0.1truecm}\vspace*{0.05cm} If $\nabla \varphi(\tilde x^k)^Td^k < 0$ then
\item[]$\,\,\,3$\hspace*{0.9truecm}\vspace*{0.05cm} Set $\alpha = 1$
\item[]$\,\,\,4$\hspace*{0.9truecm}\vspace*{0.05cm} While $\varphi(\tilde x^k+ \alpha d^k )> \bar \varphi^k +\gamma\, \alpha\, \nabla \varphi(\tilde x^k)^T d^k$
\item[]$\,\,\,5$\hspace*{1.7truecm}\vspace*{0.05cm} Set $\alpha = \delta \alpha$
\item[]$\,\,\,6$\hspace*{0.9truecm}\vspace*{0.05cm} End while
\item[]$\,\,\,7$\hspace*{0.1truecm}\vspace*{0.05cm} Else
\item[]$\,\,\,8$\hspace*{0.9truecm}\vspace*{0.05cm} Set $\alpha = 0$
\item[]$\,\,\,9$\hspace*{0.1truecm}\vspace*{0.05cm} End if
\item[]$10$\hspace*{0.1truecm}\vspace*{0.05cm} Set $\alpha^k = \alpha$
\par\vspace*{0.1cm}
}
\end{algorithmic}
\end{algorithm}

The search direction $d^k$ at $\tilde x^k$ (see line~$7$ of Algorithm~\ref{alg:ASl1})
is made of two subvectors: $d^k_{A_{\ell_1}^k}$ and $d^k_{N_{\ell_1}^k}$.
Since we do not want to move the variables in $A_{\ell_1}^k$, we simply set $d^k_{A_{\ell_1}^k}=0$.
For $d^k_{N_{\ell_1}^k}$, we compute a projected gradient direction in a properly defined manifold.
In particular, let $\mathcal B_{N_{\ell_1}^k}$ be the set defined as
\begin{equation}\label{subspaceN}
\mathcal B_{N_{\ell_1}^k} = \{x\in \R^n \colon \|x\|_1 \le \tau, \, x_i = 0, \, \forall  i\notin N_{\ell_1}^k\}
\end{equation}
and let $P(\cdot)_{\mathcal B_{N_{\ell_1}^k}}$ denote the projection onto the $\mathcal B_{N_{\ell_1}^k}$.
We also define
\begin{equation}\label{x_proj}
\hat x^k = P\bigl(\tilde  x^k-m^k\nabla \varphi(\tilde x^k)\bigr)_{\mathcal B_{N_{\ell_1}^k}},
\end{equation}
where $0 < \underline m \le m^k \le \overline m < \infty$
and with $\underline m$, $\overline m$ being two constants.
Then, $d^k_{N_{\ell_1}^k}$ is defined as
\begin{equation}\label{d}
d^k_{N_{\ell_1}^k} = \hat x^k-\tilde x^k.
\end{equation}

In the practical implementation of \ASl1, we compute the coefficient $m^k$ so that the resulting search direction is a spectral (or Barzilai-Borwein) gradient direction. This choice will be described in Section~\ref{sec:results}.

\subsection{Global convergence analysis}
In order to prove global convergence of \ASl1\ to stationary points, we need some intermediate results.
We first point out a property of our search directions, using standard results on projected directions.
\begin{lemma}\label{lemma:dbbnm}
Let Assumption~\ref{ass:eps_l1} hold and let $\{x^k\}$ be the sequence of points produced by~\ASl1.
At every iteration $k$, we have that
\begin{equation}\label{eq:derdirbound}
 \nabla \varphi(\tilde x^k)^Td^k \le \kblue{ - \frac{1}{\overline m}} \|d^k\|^2
\end{equation}
and $\{d^k\}$ is a bounded sequence.
\end{lemma}

\begin{proof}
Using the properties of the projection, at very iteration $k$ we have
\[
(\tilde x^k - m^k\nabla \varphi(\tilde x^k)-\hat x^k)^T(x -\hat x^k) \le 0, \quad \forall x\in \mathcal B_{N_{\ell_1}^k},
\]
with $B_{N_{\ell_1}^k}$ and $\hat x^k$ being defined as in~\eqref{subspaceN} and~\eqref{x_proj}, respectively.
Choosing $x=\tilde x^k$ in the above inequality and recalling the definition of $d^k$ given in~\eqref{d}, we get
\[
\nabla\varphi(\tilde x^k)^Td^k \le - \frac{1}{m^k} \|d^k\|^2.
\]
Since \kblue{$m^k \leq \overline m$}, for all $k$ we obtain~\eqref{eq:derdirbound}.

Furthermore, from the property of the projection we have that
\[
\|d^k\| = \| P(\tilde x^k - m^k\nabla \varphi(\tilde x^k)) - \tilde x^k\| \le m^k \|\nabla \varphi(\tilde x^k)\|.
\]
Since $m^k \le \overline m$ and $\{\nabla \varphi(\tilde x^k)\}$ is bounded, it follows that $\{d^k\}$ is bounded.

\end{proof}

We now prove that the
sequence~$\{\bar \varphi^k\}$ converges.
\begin{lemma}\label{lemma:barphi}
Let Assumption~\ref{ass:eps_l1} hold and let $\{x^k\}$ be the sequence of points produced by~\ASl1.
Then, the sequence $\{ \bar \varphi^k\}$ is non-increasing and converges to a value $\bar \varphi$.
\end{lemma}

\begin{proof}
First note that the definition of~$\bar \varphi^k$ ensures~$\bar
\varphi^k\le \varphi(\tilde x^0)$ and hence~$\varphi(\tilde x^k)\le \varphi(\tilde x^0)$ for all~$k$.
Moreover, we have that
\begin{equation*}
  \bar \varphi^{k+1} = \max\limits_{0\le i\le \min\{n_m,k+1\}} \varphi(\tilde x^{k + 1
    -i}) \le \max\{ \bar \varphi^k , \varphi(\tilde x^{k+1})\}.
\end{equation*}
Since $\varphi(\tilde x^{k+1}) \le \bar \varphi^k$ by the definition of the line search, we
derive $\bar \varphi^{k+1} \le \bar \varphi^k$, which proves that the sequence $\{
\bar \varphi^k\}$ is non-increasing. This sequence is
bounded from below by the minimum of~$\varphi$ over the unit simplex and hence converges. 
\end{proof}

The next intermediate result shows that the distance between 
$\{x^k\}$ and $\{\tilde x^k\}$ converges to zero and that
 the sequences
$\{\varphi(x^k)\}$ and $\{\varphi(\tilde x^k)\}$ converge to the same point, using similar arguments as in~\cite{grippo:1986}.
\begin{proposition}\label{prop:limgraddirnm-0}
	Let Assumption~\ref{ass:eps_l1} hold and let $\{x^k\}$ be the sequence of points produced by~\ASl1.
	Then,
	\begin{gather}
	\lim_{k \to \infty} \|\tilde x^k-x^k\| = 0, \label{simpl_x_tilde_to_xnm} \\
	\lim_{k\to\infty} \varphi(\tilde x^k) = \lim_{k\to\infty} \varphi( x^k) =\bar \varphi. \label{eq:varphixk}
	\end{gather}
\end{proposition}
\begin{proof}
	For each~$k\in\N$, choose $l(k)\in\{k - \min(k,n_m),\dots,k\}$ such that $\bar \varphi^k = \varphi(\tilde x^{l(k)})$.
	From Proposition~\ref{prop1_l1}
	we can write
	\begin{equation}\label{eq:xlk}
		\varphi(\tilde x^{l(k)})\le \varphi(x^{l(k)}) - CL \norm{\tilde x^{l(k)}-x^{l(k)}}^2.
	\end{equation}
	Furthermore, from the instructions of the line search and the fact that the sequence $\{\varphi(\tilde x^{l(k)})\}$ is non-increasing, for all $k \ge1 $we have
	\[\varphi(x^{l(k)}) \leq
	\varphi(\tilde x^{l(k -1)})  + \gamma \alpha^{l(k) -1} \nabla \varphi(\tilde x^{l(k) -1})^T d^{l(k) -1},\]
	and then,
	\begin{equation}\label{eq:maj}
		\varphi(\tilde x^{l(k)}) \le  \varphi(\tilde x^{l(k -1)})  + \gamma \alpha^{l(k) -1} \nabla \varphi(\tilde x^{l(k) -1})^T d^{l(k) -1} - CL \norm{\tilde x^{l(k)}-x^{l(k)}}^2.
	\end{equation}
	Since $\{\varphi(\tilde x^{l(k)})\}$ converges to $\bar \varphi$, we have that~\eqref{eq:xlk} and~\eqref{eq:maj} imply
	\begin{align}
		& \lim_{k\to\infty}  \|\tilde x^{l(k)} - x^{l(k)}\| = 0, \label{xtildexk} \\
		&  \lim_{k\to\infty} \alpha^{l(k)-1} \nabla \varphi(\tilde x^{l(k) -1})^T d^{l(k) -1} =0 \notag.
	\end{align}
	Furthermore, from Lemma~\ref{lemma:dbbnm} we have
	\[
	\nabla \varphi(\tilde x^{l(k) -1})^T d^{l(k) -1} \leq \kblue{-  \frac{1}{\overline m}}  \|d^{l(k) -1}\|^2,
	\]
	and then the following limit holds:
	\begin{equation}\label{(8)}
		\lim_{k\to\infty}  \alpha^{l(k) -1} \|d^{l(k) -1}\| = 0.
	\end{equation}
	Considering that
	$x^{l(k)} = \tilde x^{l(k)-1} + \alpha^{l(k) -1} d^{l(k) -1}$,
	\eqref{(8)} implies
	\begin{equation*}\label{x-xtildelk}
		\lim_{k\to\infty}   \|\tilde x^{l(k)-1} - x^{l(k)}\| = 0.
	\end{equation*}
	Furthermore, from the triangle inequality, we can write
	\[
	\|\tilde x^{l(k)-1} - \tilde x^{l(k)}\| \leq \|\tilde x^{l(k)-1} - x^{l(k)}\|+\|x^{l(k)} - \tilde x^{l(k)}.\|
	\]
	Then,
	\begin{equation}\label{diff2iniz}
	\lim_{k\to\infty}   \|\tilde x^{l(k)-1} - \tilde x^{l(k)}\| = 0
	\end{equation}
	and in particular, from the uniform continuity of $\varphi$  over $\{x\in \R^n : \|x\|_1 \leq \tau\}$, we have
	\begin{equation}\label{eq:phiind}
		\lim_{k\to\infty}\varphi(\tilde x^{l(k)-1}) = \lim_{k\to\infty}\varphi(\tilde x^{l(k)}) = \bar\varphi.
	\end{equation}
	Let
	\[
	\hat l(k) = l(k+n_m+2).
	\]
	We show by induction that, for any given $j\geq 1$,
	\begin{align}
	&	\lim_{k\to\infty}   \|x^{\hat l(k) -{(j-1)}}-\tilde x^{\hat l(k) -{(j-1)}}\| = 0, \label{diff1} \\
	&	\lim_{k\to\infty} \|\tilde x^{\hat l(k) -{(j-1)}}-\tilde x^{\hat l(k) -{j}}\| = 0, \label{diff2}\\		
	&\lim_{k\to\infty} \varphi(\tilde x^{\hat l(k)-j}) = \lim_{k\to\infty} \varphi(\tilde x^{l(k)}). \label{varphij}	
	\end{align}	
	If $j=1$, since $\{\hat l(k)\}\subset \{l(k)\}$~we have that~\eqref{diff1}, \eqref{diff2} and~\eqref{varphij}
	follow from~\eqref{xtildexk}, \eqref{diff2iniz} and~\eqref{eq:phiind}, respectively.
	
	Assume now that~\eqref{diff1}, \eqref{diff2} and~\eqref{varphij} hold for a given $j$.
	Then, reasoning as in the beginning of the proof,
	from the instructions of the line search and considering that $\{\varphi(\tilde x^{l(k)})\}$ is non-increasing, we can write
	\[
	\varphi(\tilde x^{\hat l(k)-j}) \le \varphi(x^{\hat l(k)-j}) - CL \norm{\tilde x^{\hat l(k)-j}-x^{\hat l(k)-j}}^2
	\]
	and
	\[\varphi(x^{\hat l(k) -j})\leq \varphi(\tilde x^{\hat l(k-(j+1))}) + \gamma \alpha^{\hat l(k) -(j+1)} \nabla \varphi(\tilde x^{\hat l(k) -(j+1)})^\top d^{\hat l(k) -(j+1)}.\]
	Therefore we get
	\[
	\begin{split}
	\varphi(\tilde x^{\hat l(k)-j})
	\leq & \varphi(\tilde x^{\hat l(k -(j+1))}) + \gamma \alpha^{\hat l(k) -(j+1)} \nabla \varphi(\tilde x^{\hat l(k)-(j+1)})^T d^{\hat l(k)-(j+1)} + \\
	& -	CL \norm{\tilde x^{l(k)-j}-x^{l(k)-j}}^2,
	\end{split}
	\]
	so that
	\begin{align}
		& \lim_{k\to\infty} \alpha^{\hat l(k)-(j+1)} \nabla \varphi(\tilde x^{\hat l(k)-(j+1)})^T d^{\hat l(k)-(j+1)}=0,  \label{diff2-1} \\
		& \lim_{k\to\infty} \norm{\tilde x^{l(k)-j}-x^{l(k)-j}}=0 \label{limalphagraddirj} .
	\end{align}
	The limit in~\eqref{limalphagraddirj} implies~\eqref{diff1} for $j+1$.
	The properties of the direction stated in Lemma~\ref{lemma:dbbnm}, combined with~\eqref{diff2-1},  ensure that 
	\begin{equation}\label{9j+1}
		\lim_{k\to\infty}  \alpha^{\hat l(k) -(j+1)} \|d^{\hat l(k) -(j+1)}\| = 0.
	\end{equation}
	Furthermore, since
	$x^{\hat l(k)-j} = \tilde x^{\hat l(k)-(j+1)} + \alpha^{\hat l(k) -(j+1)} d^{\hat l(k) -(j+1)}$,
	we have that~\eqref{9j+1} implies
	\begin{equation*}\label{x-xtildelj+1}
		\lim_{k\to\infty}   \|\tilde x^{\hat l(k)-(j+1)} - x^{\hat l(k)-j}\| = 0.
	\end{equation*}
	Using the triangle inequality, we can write
	\[\|\tilde x^{\hat l(k)-(j+1)} - \tilde x^{\hat l(k)-j}\| \leq \|\tilde x^{\hat l(k)-(j+1)} - x^{\hat l(k)-j}\|+\|x^{\hat l(k)-j}
	- \tilde x^{\hat l(k)-j}\|.\]
	Then,
	\[
		\lim_{k\to\infty}   \|\tilde x^{\hat l(k)-(j+1)} - \tilde x^{\hat l(k)-j}\| = 0
	\]
	and in particular, from the uniform continuity of $\varphi$  over $\{x\in \R^n : \|x\|_1 \leq \tau\}$, we can write
	\[
	\lim_{k\to\infty}\varphi(\tilde x^{\hat l(k)-(j+1)}) = \lim_{k\to\infty}\varphi(\tilde x^{\hat l(k)-j}) = \bar\varphi.
	\]
	Thus we conclude that~\eqref{diff2} and~\eqref{varphij}  hold for any given $j\geq 1$.
	%
	Recalling that
	\begin{align*}
		& \hat l(k)-(k+1) = l(k+n_m+2)-(k+1)\le n_m+1, \\
	&	\|\tilde x^{k+1} - \tilde x^{\hat l(k)}\|\leq  \sum_{\kblue{j = k+1}}^{\hat l(k)-1} \|\tilde x^{j+1} - \tilde x^{j}\|,
	\end{align*} 
%
	we have that~\eqref{diff2}  implies
		\begin{equation}\label{diff1-fin}
	\lim_{k\to\infty}  \|\tilde x^{k+1} - \tilde x^{\hat l(k)}\| = 0.
	\end{equation}
	Furthermore, since
	\[
	\|x^{k+1} - \tilde x^{\hat l(k)}\|
	\leq \|x^{k+1} - \tilde x^{k+1}\| +\|\tilde x^{k+1} - \tilde x^{\hat l(k)}\|,
	\]
	from \eqref{diff1-fin} and \eqref{diff1} we have
			\begin{equation}\label{diff2-fin}
	\lim_{k\to\infty}  \|x^{k+1} - \tilde x^{\hat l(k)}\| = 0.
	\end{equation}
	Since $\{\varphi(\tilde x^{\hat l(k)})\}$ has a limit, from the uniform continuity of  $\varphi$  over $\{x\in \R^n : \|x\|_1 \leq \tau\}$, \eqref{diff2-fin} and \eqref{diff1-fin}  it follows that
	\begin{equation}\nonumber
		\lim_{k\to\infty} \varphi(x^{k+1}) =\lim_{k\to\infty} \varphi(x^k) = \lim_{k\to\infty} \varphi(\tilde x^{\hat l(k)}) = \bar \varphi
	\end{equation}
	and
	\[\lim_{k\to\infty} \varphi(\tilde x^{k+1}) = \lim_{k\to\infty} \varphi(\tilde x^k) = \lim_{k\to\infty} \varphi(\tilde x^{\hat l(k)}) = \bar \varphi,\]
	proving~\eqref{eq:varphixk}.
	From the instructions of the algorithm and Proposition~\ref{prop1_l1}, we can write
	\[
	 \varphi(\tilde x^k)\le \varphi(x^k) - CL \norm{\tilde x^k-x^k}^2,
	\]
	and then from~\eqref{eq:varphixk} we have that~\eqref{simpl_x_tilde_to_xnm} holds.
	 
\end{proof}

The following proposition states that the directional derivative $\nabla \varphi(\tilde x^k)^Td^k$
tends to zero.

\begin{proposition}\label{prop:limgraddirnm}
	Let Assumption~\ref{ass:eps_l1} hold and let $\{x^k\}$ be the sequence of points produced by~\ASl1.
	Then,
	\begin{equation}
		\lim_{k\to\infty} \nabla \varphi(\tilde x^k)^Td^k = 0. \label{simpl_gd_to_zeronm}
	\end{equation}
\end{proposition}
\begin{proof}
	To prove~\eqref{simpl_gd_to_zeronm}, assume by contradiction that it does not hold.
	Lemma~\ref{lemma:dbbnm} implies that the sequence $\{\nabla \varphi(\tilde x^k)^Td^k\}$ is bounded, so that
	there must exist an infinite set $K\subseteq \mathbb N$ such that
	\begin{gather}
		\nabla \varphi(\tilde x^k)^Td^k < 0, \quad \forall k \in K, \label{gd_negnm} \\
		\lim_{k \to \infty, \, k \in K} \nabla \varphi(\tilde x^k)^Td^k = -\eta < 0, \label{lim_gd_contrnm}
	\end{gather}
	for some real number $\eta>0$.
	Taking into account~\eqref{simpl_x_tilde_to_xnm} and the fact that the feasible set is compact, without loss of generality we can assume that
	both $\{x^k\}_K$ and $\{\tilde x^k\}_K$ converge to a feasible point $x^*$ (passing into a further subsequence if necessary). Namely,
	\begin{equation}\label{convxknm}
		\lim_{k\rightarrow \infty, \, k\in K} x^k = \lim_{k\rightarrow \infty, \, k\in K} \tilde x^k = x^*.
	\end{equation}
	Moreover, since the number of possible different choices of $A^k$ and $N^k$ is finite,
	without loss of generality we can also assume that
	\[
	A^k = \hat A, \quad N^k = \hat N, \quad \forall k \in K,
	\]
	and, using the fact that $\{d^k\}$ is a bounded sequence, that
	\begin{equation} \label{lim_dnm}
		\lim_{k\to\infty, \, k\in K} d^k = \bar d
	\end{equation}
	(passing again into a further subsequence if necessary).
	From~\eqref{lim_gd_contrnm}, \eqref{convxknm}, \eqref{lim_dnm} and the continuity of $\nabla \varphi$, we can write
	\begin{equation}\label{lim_gd_contr_2nm}
		\nabla \varphi(x^*)^T \bar d = -\eta < 0.
	\end{equation}
	
	Taking into account~\eqref{gd_negnm}, from the instructions of \ASl1 we have that, at every iteration $k \in K$, a
	non-monotone Armijo line search is carried out
	(see line~2 in Algorithm~\ref{alg:ls_as}) and a value $\alpha^k \in (0,1]$ is computed such that
	\[
		\varphi(x^{k+1}) \le \varphi(\tilde x^{l(k)}) + \gamma \, \alpha^k \, \nabla \varphi(\tilde x^k)^T d^k,
	\]
	or equivalently,
	\[
	\varphi(\tilde x^{l(k)}) - \varphi(x^{k+1}) \ge \gamma \, \alpha^k \, \abs{\nabla \varphi(\tilde x^k)^T d^k}.
	\]
	From
	\eqref{eq:varphixk}, the left-hand side of the above inequality converges to zero for $k \to \infty$,
	hence
	\[
		\lim_{k \to \infty, \, k \in K} \alpha^k\, \abs{\nabla \varphi(\tilde x^k)^T d^k} = 0.
	\]
	Using~\eqref{lim_gd_contrnm}, we obtain that $\displaystyle{\lim_{k \to \infty, \, k \in K} \alpha^k = 0}$.
	It follows that there exists $\bar k \in K$ such that
	\begin{equation*}
		\alpha^k < 1, \quad \forall k \ge \bar k, \, k \in K.
	\end{equation*}
	From the instructions of the line search procedure, this means that $\forall k \ge \bar k, \, k \in K$
	\begin{equation}\label{simpl_alphadeltanm}
		\varphi\Bigl( \tilde x^k + \frac{\alpha^k}{\delta} d^k \Bigr) >
		\varphi(\tilde x^{l(k)}) + \gamma \, \frac{\alpha^k}{\delta}\, \nabla \varphi(\tilde x^k)^T d^k
		\geq \varphi(\tilde x^k) + \gamma \, \frac{\alpha^k}{\delta}\, \nabla \varphi(\tilde x^k)^T d^k.
	\end{equation}
	Using the mean value theorem, $\xi^k\in (0,1)$ exists such that
	\begin{equation}\label{simpl_meanvaltheonm}
		\varphi\Bigl( \tilde x^k+ \frac{\alpha^k}{\delta} d^k \Bigr)
		= \varphi(\tilde x^{l(k)}) + \frac{\alpha^k}{\delta}\nabla \varphi\Bigl( \tilde x^k + \xi^k\frac{\alpha^k}{\delta} d^k \Bigr)^T d^k,
		\quad \forall k \ge \bar k, \, k \in K.
	\end{equation}
	In view of~\eqref{simpl_alphadeltanm} and~\eqref{simpl_meanvaltheonm}, we can write
	\begin{equation}\label{simpl_meanvaltheo2nm}
		\nabla \varphi\Bigl( \tilde x^k + \xi^k \frac{\alpha^k}{\delta} d^k \Bigr)^T d^k > \gamma\, \nabla \varphi(\tilde x^k)^T d^k,
		\quad \forall k \ge \bar k, \, k \in K.
	\end{equation}
	From~\eqref{convxknm}, and exploiting the fact that $\{\xi^k\}_K$, $\{\alpha^k\}_K$ and $\{d^k\}_K$ are bounded sequences, we get
	\begin{equation*}
		\lim_{k \to \infty, \, k \in K} \tilde x^k + \xi^k \frac{\alpha^k}{\delta} d^k = \lim_{k \to \infty, \, k \in K} \tilde x^k = x^*.
	\end{equation*}
	Therefore, taking the limits in~\eqref{simpl_meanvaltheo2nm} we obtain that
	$\nabla \varphi(x^*)^T \bar d \ge \gamma\, \nabla \varphi(x^*)^T \bar d$, or equivalently,
	$(1-\gamma) \nabla \varphi(x^*)^T \bar d \ge 0$.
	Since $\gamma \in (0,1)$, we get a contradiction with~\eqref{lim_gd_contr_2nm}. 
\end{proof}

We are finally able to state the main convergence result.
\begin{theorem}\label{th:convresnm}
Let Assumption~\ref{ass:eps_l1} hold and let $\{x^k\}$ be the sequence of points produced by~\ASl1.
Then, every limit point $x^*$ of $\{x^k\}$ is a stationary point of problem~\eqref{probl1}.
\end{theorem}

\begin{proof}
From Definition \ref{def:stat}, we can characterize stationarity using condition \eqref{eq:stat}.
In particular, we can define the following continuous functions $\Psi_i(x)$ to measure the stationarity violation at a feasible point $x$:
\[
\Psi_i(x) = \max\{0, -\nabla\varphi(x)^T(\tau \,e_i - x) , -\nabla\varphi(x)^T(-\tau \,e_i - x)  \}, \quad i = 1,\ldots,n,
\]
so that a feasible point $x$ is stationary if and only if $\Psi_i(x) = 0$, $i = 1,\ldots,n$.

Now, let $x^*$ be a limit point of $\{x^k\}$ and let $\{x^k\}_K$, $K \subseteq \mathbb N$, be a subsequence converging to $x^*$. 
Namely,
\begin{equation}\label{convxk2nm}
\lim_{k\rightarrow \infty, \, k\in K} x^k =  x^*.
\end{equation}
Note that $x^*$ exists, as $\{x^k\}$ remains in the compact set $\{x\in \R^n | \norm{x}_1 \le \tau\}$. 
Since the number of possible different choices of $A^k$ and $N^k$ is finite,
without loss of generality we can assume that
\[
A^k = \hat A, \quad N^k = \hat N, \quad \forall k \in K
\]
(passing into a further subsequence if necessary).

By contradiction, assume that $x^*$ is non-stationary, that is, an index $\nu \in \{1,\ldots,n\}$ exists such that
\begin{equation}\label{contrnm}
\Psi_\nu(x^*) > 0.
\end{equation}

First, suppose that $\nu \in \hat A$. Then, from the expressions~\eqref{act_set_l1},  
we can write
\[
0 \le x^k_\nu \le \epsilon \tau\nabla \varphi(x^k)^T(\tau e_\nu - x^k)
\qquad \text{or} \qquad
0 \ge x^k_\nu \ge \epsilon \tau\nabla \varphi(x^k)^T(\tau e_\nu + x^k),
\]
so that $\Psi_\nu(x^k) = 0$, for all $k \in \bar K$.
Therefore, from~\eqref{convxk2nm}, the continuity of $\nabla \varphi$ and the continuity of the functions $\Psi_i$, we get $\Psi_\nu(x^*) = 0$,
contradicting~\eqref{contrnm}.

Then, $\nu$ necessarily belongs to $\hat N$. Namely, $x^*$ is non-stationary over $\mathcal B_{N_{\ell_1}^k}$,
with $\mathcal B_{N_{\ell_1}^k}$ defined as in~\eqref{subspaceN}.
This means that
\begin{equation}\label{stat_contrnm}
x^* \ne P\bigl(x^* - \underline m \nabla \varphi(x^*)\bigr)_{\mathcal B_{N_{\ell_1}^k}}.
\end{equation}
Using Proposition~\ref{prop:limgraddirnm} and Lemma~\ref{lemma:dbbnm}, we have that $\lim_{k\to\infty, \, k\in K} \|d^k\| = 0$,
that is, recalling the definition of $d^k$ given in~\eqref{x_proj}--\eqref{d},
\[
\lim_{k\to\infty, \, k\in K} \Biggl\| \tilde x^k - P\bigl(\tilde  x^k- m^k\nabla \varphi(\tilde x^k)\bigr)_{\mathcal B_{N_{\ell_1}^k}} \Biggr\| = 0.
\]
From the properties of the projection we have that
\[\Biggl\| \tilde x^k - P\bigl(\tilde  x^k- m^k\nabla \varphi(\tilde x^k)\bigr)_{\mathcal B_{N_{\ell_1}^k}} \Biggr\|
\geq
\Biggl\| \tilde x^k - P\bigl(\tilde  x^k- \underline m\nabla \varphi(\tilde x^k)\bigr)_{\mathcal B_{N_{\ell_1}^k}}, \Biggr\|
\]
so that the following holds
\[
\lim_{k\to\infty, \, k\in K} \Biggl\| \tilde x^k - P\bigl(\tilde  x^k- \underline m\nabla \varphi(\tilde x^k)\bigr)_{\mathcal B_{N_{\ell_1}^k}} \Biggr\| = 0.
\]
Using~\eqref{convxk2nm}, the continuity of the projection and taking into
account~\eqref{simpl_x_tilde_to_xnm} in Proposition~\ref{prop:limgraddirnm}, we obtain
\[
 \Biggl\|x^* - P\bigl(x^* -\underline m  \nabla \varphi(x^*)\bigr)_{\mathcal B_{N_{\ell_1}^k}} \Biggr\| = 0.
\]
This contradicts~\eqref{stat_contrnm}, leading to the desired result. 
\end{proof}

\section{Numerical results}\label{sec:results}
In this section, we show the practical performances of \ASl1\ on two classes of problems frequently arising in data science and machine learning that can be formulated as problem \eqref{probl1}:
\begin{itemize}
\item LASSO problems~\cite{tibshirani1996regression}, 
where
\begin{equation}\label{LASSO_obj}
\varphi(x) = \|Ax-b\|^2,
\end{equation}
for given matrix $A\in \R^{m\times n}$ and vector $b\in \R^m$;
\item $\ell_1$-constrained logistic regression problems, where
\begin{equation}\label{logreg_obj}
\varphi (x) = \sum_{i=1}^l \log(1 + \exp(-y_i x^T a_i)),
\end{equation}
with given vectors $a_i$ and scalars $y_i \in \{1,-1\}$, $i=1,\ldots,l$.
\end{itemize}


In our implementation of \ASl1, we used a non-monotone line search with memory length $n_m = 10$ (see Algorithm~\ref{alg:ls_as})
and a spectral (or Barzilai-Borwein) gradient direction
for the variables in $N_{\ell_1}^k$.
In particular, the coefficient $m^k$ appearing in~\eqref{x_proj}
was set to $1$ for $k=0$ and, for $k \ge 1$,
we employed the following formula, adapting the strategy used in~\cite{andreani2010second,birgin2002large,cristofari2020total}:
\[
m^k =
\begin{cases}
\max\{\underline m, \, m^k_a\}, \quad & \text{if } 0 < m^k_a < \overline m, \\[1.1ex]
\max\bigl\{\underline m, \, \min\{\overline m, \, m^k_b\}\bigr\}, \quad & \text{if } m^k_a \ge \overline m, \\[1.1ex]
\max\Biggl\{\underline m, \, \min\biggl\{1, \, \dfrac{\norm{\nabla_{N_{\ell_1}^k} \varphi(\tilde x^k)}}{\norm{\tilde x^k_{N_{\ell_1}^k}}}\biggr\}\Biggr\}, \quad & \text{if } m^k_a \le 0,
\end{cases}
\]
where $\underline m = 10^{-10}$, $\overline m = 10^{10}$,
$m^k_a = \dfrac{(s^{k-1})^T y^{k-1}}{\norm{s^{k-1}}^2}$, $m^k_b = \dfrac{\norm{y^{k-1}}^2}{(s^{k-1})^T y^{k-1}}$,
$s^{k-1} = \tilde x^k_{N_{\ell_1}^k}-\tilde x^{k-1}_{N_{\ell_1}^k}$ and $y^{k-1} = \nabla_{N_{\ell_1}^k} \varphi(\tilde x^k) - \nabla_{N_{\ell_1}^k} \varphi(\tilde x^{k-1})$.

The $\epsilon$ parameter appearing in the active-set estimate~\eqref{act_set_l1}
should satisfy Assumption~\ref{ass:eps_l1} to guarantee the descent property established in Proposition~\ref{prop1_l1} and the convergence of the algorithm.
Since the Lipschitz constant $L$ is in general unknown, we approximate $\epsilon$ following the same strategy as in~\cite{cristofari:2017,cristofari:2020,desantis2016fast}, where similar estimates are used.
Starting from $\epsilon = 10^{-6}$, we update its value along the iterations, reducing it whenever the expected decrease in the objective, stated in Proposition~\ref{prop1_l1}, is not obtained.

In our experiments, we implemented \ASl1\ in Matlab and compared it with the two following first-order methods, implemented in Matlab as well:
\begin{itemize}
\item a spectral projected gradient method with non-monotone line search, which will be referred to as \texttt{NM-SPG}, downloaded from Mark Schmidt's webpage \url{https://www.cs.ubc.ca/~schmidtm/Software/minConf.html};
\item the away-step Frank-Wolfe method with Armijo line search~\cite{bomze2019first,bomze:2019}, which will be referred to as \texttt{AFW}\footnote{\texttt{AFW} was run by reformulating~\eqref{probl1} as an optimization problem over the unit simplex, exploiting the fact that the feasible set is a convex combination of the vectors $\pm \tau e_i$, $i=1,\ldots,n$.}.
\end{itemize}

For every considered problem, we set the starting point \kblue{equal to the origin} and
we first run \ASl1, stopping when
\[
\|x^k - P\bigl(x^k-\nabla \varphi(x^k)\bigr)_{\ell_1}\| \le 10^{-6},
\]
where $P(\cdot)_{\ell_1}$ denotes the projection onto the $\ell_1$-ball.
Then, the other methods were run with the same starting point and were
stopped at the first iteration $k$ such that
\[
\varphi(x^k) \le f^* + 10^{-6}(1+|f^*|),
\]
with $f^*$ being the objective value found by $\ASl1$.
A time limit of $3600$ seconds was also included in all the considered methods.

In \texttt{NM-SPG}, we used the default parameters (except for those concerning the stopping condition).
Moreover, in \ASl1 and \texttt{NM-SPG} we employed the same projection algorithm~\cite{condat2016fast}, downloaded from Laurent Condat's webpage \url{https://lcondat.github.io/software.html}.

In all codes, we made use of the Matlab \textit{sparse} operator to compute $\varphi(x)$ and $\nabla \varphi(x)$, in order to exploit the problem structure and save computational time.
The experiments were run on an Intel Xeon(R) CPU E5-1650 v2 @ 3.50GHz with 12 cores and 64 Gb RAM.

\kblue{The \ASl1\ software is available at \url{https://github.com/acristofari/as-l1}.}

\subsection{Comparison on LASSO instances}
We considered $10$ artificial instances of LASSO problems, where the objective function $\varphi (x)$ takes the form of~\eqref{LASSO_obj}.
Each instance was created by first generating a matrix $A \in \R^{m \times n}$ with elements randomly drawn from a uniform distribution on the interval $(0,1)$, using $n=2^{15}$ and $m = n/2$.
Then, a vector $x^*$ was generated with all zeros, except for $\text{round}(0.05m)$ components, which were randomly set to $1$ or $-1$.
Finally, we set $b = Ax^* + 0.001 v$, where $v$ is a vector with elements randomly drawn from the standard normal distribution,
and the $\ell_1$-sphere \kblue{radius} $\tau$ was set to $0.99\|x^*\|_1$.

The detailed comparison on the LASSO instances is reported in Table~\ref{tab:LASSO}.
For each instance and each algorithm, we report the final objective function value found, the CPU time needed to satisfy the stopping criterion and the percentage of zeros in the final solution, with a tolerance of $10^{-5}$.
In case an algorithm reached the time limit on an instance, we consider as final solution and final objective value those related to the last iteration performed. \texttt{NM-SPG} reached the time limit on all instances, being very far from $f^*$ on $6$ instances out of $10$, with a difference of even two order of magnitude.
\texttt{AFW} gets the same solutions as those obtained by $\ASl1$, being however
an order of magnitude slower than $\ASl1$.

The same picture is given by Figure~\ref{fig:plot_LASSO}, where we report the average optimization error $f(x^k)-f_{\text{best}}$ over the $10$ instances, with $f_{\text{best}}$ being the minimum objective value found by the algorithms.
We can notice that $\ASl1$ clearly outperforms the other two methods.

\begin{table}
\caption{Comparison on $10$ LASSO instances.
For each method, the first column (Obj) indicates the final objective value, the second column (CPU time) indicates the required time in seconds, where a star means that the time limit of $3600$ seconds was reached, and the third column (\%zeros) indicates the percentage of zeros in the final solution, with a tolerance of $10^{-5}$. \kblue{For each problem, the fastest algorithm is highlighted in bold.}}\label{tab:LASSO}
\centering
{\scalebox{.92}{
\begin{tabular}{| c c c | c c c | c c c |}
\hline
\multicolumn{3}{|c|}{\ASl1} & \multicolumn{3}{|c|}{\texttt{NM-SPG}} & \multicolumn{3}{|c|}{\texttt{AFW}} \bigstrut[t]\\
Obj & CPU time & \%zeros & Obj & CPU time & \%zeros & Obj & CPU time & \%zeros \bigstrut[t] \bigstrut[b] \\
\hline
$54.20$ & $\textbf{315.85}$ & $97.49$ & $54.20$ & $*$ & $97.49$ & $54.20$ & $2762.90$ & $97.49$ \bigstrut[t] \\
$52.32$ & $\textbf{366.90}$ & $97.50$ & $5823.87$ & $*$ & $80.69$ & $52.32$ & $3046.89$ & $97.50$ \\
$53.95$ & $\textbf{449.67}$ & $97.50$ & $1040.99$ & $*$ & $85.12$ & $53.95$ & $3023.94$ & $97.50$ \\
$54.04$ & $\textbf{292.83}$ & $97.50$ & $2215.88$ & $*$ & $83.45$ & $54.04$ & $3050.17$ & $97.50$ \\
$52.98$ & $\textbf{330.65}$ & $97.50$ & $841.57$ & $*$ & $85.21$ & $52.98$ & $2798.97$ & $97.50$ \\
$53.54$ & $\textbf{387.79}$ & $97.50$ & $53.56$ & $*$ & $97.50$ & $53.54$ & $3006.38$ & $97.50$ \\
$52.71$ & $\textbf{806.80}$ & $97.50$ & $3927.10$ & $*$ & $82.23$ & $52.71$ & $2837.90$ & $97.50$ \\
$53.58$ & $\textbf{580.89}$ & $97.50$ & $4108.25$ & $*$ & $81.93$ & $53.58$ & $2768.45$ & $97.50$ \\
$52.61$ & $\textbf{402.03}$ & $97.50$ & $1750.54$ & $*$ & $83.37$ & $52.61$ & $2924.38$ & $97.50$ \\
$53.36$ & $\textbf{535.10}$ & $97.50$ & $53.89$ & $*$ & $97.49$ & $53.36$ & $2948.41$ & $97.50$ \bigstrut[b] \\
\hline
\end{tabular}
}}
\end{table}

\begin{figure}[htp]
\centering
\includegraphics[scale=0.5, trim = 0cm 0cm 0cm 0cm, clip]{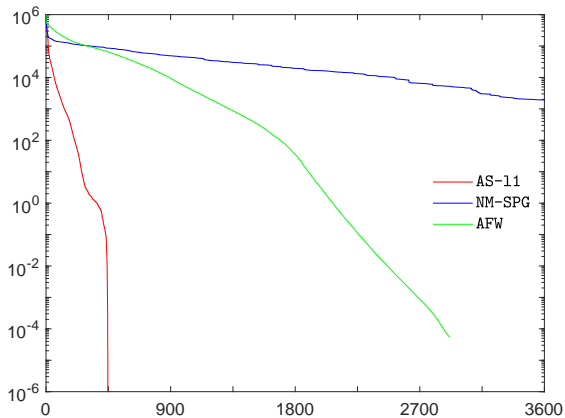}
\caption{Average optimization error over LASSO instances ($y$ axis) vs CPU time in seconds ($x$ axis). The $y$ axis is in logarithmic scale.}
\label{fig:plot_LASSO}
\end{figure}

\subsection{Comparison on logistic regression instances}
For the comparison among \ASl1, \texttt{NM-SPG} and \texttt{AFW} on $\ell_1$-constrained logistic regression problems, where the objective function $\varphi (x)$ takes the form of~\eqref{logreg_obj}, we considered $11$ datasets for binary classification from the literature,
with a number of samples $l$ between $100$ and $25,000$,
and a number of attributes $n$ between $500$ and $100,000$.
We report the complete list of datasets in Table~\ref{table:datasets_log_reg}.

\begin{table}[h]
\caption{Datasets used in the comparison on $\ell_1$-constrained logistic regression problems, where $l$ is the number of instances and $n$ is the number of attributes.}\label{table:datasets_log_reg}
\centering
\begin{tabular}{| c c c c |}
\hline
Dataset & $l$ & $n$ & Reference \bigstrut[t] \bigstrut[b] \\
\hline
Arcene (training set) & $100$ & $10,000$ & \cite{Dua:2019,guyon2004result} \bigstrut[t] \\
Dexter (training set) & $300$ & $19,999$ & \cite{Dua:2019,guyon2004result} \\
Dorothea (training set) & $800$ & $100,000$ & \cite{Dua:2019,guyon2004result} \\
Farm-ads-vect & $4,143$ & $54,877$ & \cite{Dua:2019} \\
Gisette (training set) & $6,000$ & $5,000$ & \cite{libsvm_data,Dua:2019,guyon2004result} \\
Madelon (training set) & $2,000$ & $500$ & \cite{libsvm_data,Dua:2019,guyon2004result} \\
Rcv1\_train.binary (training set) & $20,242$ & $47,236$ & \cite{libsvm_data,lewis2004rcv1} \\
Real-sim & $72,309$ & $20,958$ & \cite{libsvm_data} \\
Swarm (Aligned) & $24,016$ & $2,400$ & \cite{Dua:2019} \\
Swarm (Flocking) & $24,016$ & $2,400$ & \cite{Dua:2019} \\
Swarm (Grouped) & $24,016$ & $2,400$ & \cite{Dua:2019} \bigstrut[b] \\
\hline
\end{tabular}
\end{table}

For each dataset, we considered different values of the $\ell_1$-sphere \kblue{radius} $\tau$, that is, $0.01n$, $0.03n$ and $0.05n$. The final results are shown in Table~\ref{tab:log_reg}.
As before, for each instance and each algorithm, we report the final objective function value found, the CPU time needed to satisfy the stopping criterion and the percentage of zeros in the final solution, with a tolerance of $10^{-5}$.
In case an algorithm reached the time limit on an instance, we consider as final solution and final objective value those related to the last iteration performed.
Excluding the instance obtained from the Rev1\_train.binary dataset with $\tau = 0.05n$, the three solvers get very similar solutions on all instances, with a difference of $0.02$ at most in the final objective values.
When considering $\tau = 0.01n$, \ASl1 is 
the fastest solver on $4$ instances out of $11$.
Note that on the instance from the Farm-ads-vect dataset, \ASl1 is able to get the solution in a third of the CPU time needed by the other two solvers.
On the other instances, the CPU time needed by \ASl1 is always comparable with the one needed by the fastest solver.
Looking at the results for larger values of $\tau$, we can notice that the instances get more difficult and in general less sparse.
For $\tau = 0.03n$ and $\tau = 0.05n$, \ASl1 is 
the fastest solver on all the instances but two, those obtained from the Arcene and the Dorothea datasets, which are however addressed within $2$ seconds. On other instances, such as those built from the Real-sim and the Rev1\_train.binary datasets, \ASl1 is one or even two orders of magnitude faster with respect to \texttt{NM-SPG} and \texttt{AFW}. 

\begin{landscape}
\tiny
\begin{longtable}{| c |  c | c c c | c c c | c c c |}
\caption{Comparison on $\ell_1$-constrained logistic regression problems with different values of the sphere \kblue{radius} $\tau$.
For each method, the first column (Obj) indicates the final objective value, the second column (CPU time) indicates the required time in seconds, where a star means that the time limit of $3600$ seconds was reached, and the third column (\%zeros) indicates the percentage of zeros in the final solution, with a tolerance of $10^{-5}$. \kblue{For each problem, the fastest algorithm is highlighted in bold.}}\label{tab:log_reg}
\\ \hline
\multicolumn{1}{|c|}{\multirow{2}*{$\tau$}} &
\multicolumn{1}{|c|}{\multirow{2}*{Dataset}} &
\multicolumn{3}{|c|}{\ASl1} & \multicolumn{3}{|c|}{\texttt{NM-SPG}} & \multicolumn{3}{|c|}{\texttt{AFW}} \bigstrut[t] \\
& & Obj & CPU time & \%zeros & Obj & CPU time & \%zeros & Obj & CPU time & \%zeros \bigstrut[t] \bigstrut[b] \\
\hline
\multirow{11}*{$0.01n$} & Arcene & $0.20$ & $\textbf{8.31}$ & $99.31$ & $0.20$ & $22.34$ & $99.30$ & $0.20$ & $12.14$ & $99.32$ \bigstrut[t] \\
& Dexter & $3.88$ & $\textbf{1.09}$ & $99.41$ & $3.88$ & $1.30$ & $99.41$ & $3.88$ & $11.81$ & $99.41$ \\
& Dorothea & $0.00$ & $0.75$ & $70.09$ & $0.00$ & $\textbf{0.58}$ & $51.79$ & $0.00$ & $5.58$ & $99.82$ \\
& Farm-ads-vect & $146.38$ & $\textbf{576.62}$ & $98.39$ & $146.38$ & $1547.40$ & $98.39$ & $146.38$ & $1699.47$ & $98.41$ \\
& Gisette & $0.00$ & $\textbf{9.70}$ & $8.20$ & $0.00$ & $18.42$ & $8.22$ & $0.00$ & $354.81$ & $78.60$ \\
& Madelon & $1314.89$ & $0.60$ & $95.01$ & $1314.89$ & $\textbf{0.41}$ & $95.01$ & $1314.89$ & $0.50$ & $94.81$ \\
& Rcv1\_train.binary & $2344.19$ & $11.15$ & $98.86$ & $2344.19$ & $\textbf{7.77}$ & $98.85$ & $2344.19$ & $218.41$ & $98.87$ \\
& Real-sim & $21693.00$ & $4.33$ & $99.47$ & $21693.03$ & $\textbf{4.26}$ & $99.47$ & $21693.03$ & $104.79$ & $99.47$ \\
& Swarm (Aligned) & $2504.35$ & $51.52$ & $97.96$ & $2504.35$ & $\textbf{34.05}$ & $97.96$ & $2504.35$ & $50.48$ & $97.96$ \\
& Swarm (Flocking) & $4823.39$ & $56.99$ & $97.79$ & $4823.40$ & $\textbf{27.48}$ & $97.75$ & $4823.40$ & $29.68$ & $97.79$ \\
& Swarm (Grouped) & $5062.87$ & $47.92$ & $97.67$ & $5062.88$ & $43.53$ & $97.67$ & $5062.88$ & $\textbf{32.22}$ & $97.67$ \bigstrut[b] \\
\hline
\multirow{11}*{$0.03n$} & Arcene & $0.00$ & $1.91$ & $81.27$ & $0.00$ & $1.80$ & $83.80$ & $0.00$ & $\textbf{1.26}$ & $98.81$ \bigstrut[t] \\
& Dexter & $0.00$ & $\textbf{0.66}$ & $99.06$ & $0.00$ & $0.92$ & $99.06$ & $0.00$ & $7.01$ & $99.16$ \\
& Dorothea & $0.00$ & $0.57$ & $11.97$ & $0.00$ & $\textbf{0.35}$ & $11.90$ & $0.00$ & $5.18$ & $99.84$ \\
& Farm-ads-vect & $12.68$ & $\textbf{676.07}$ & $98.08$ & $12.68$ & $2191.57$ & $98.08$ & $12.69$ & $*$ & $98.10$ \\
& Gisette & $0.00$ & $8.91$ & $8.26$ & $0.00$ & $\textbf{7.15}$ & $7.74$ & $0.00$ & $375.52$ & $77.86$ \\
& Madelon & $1262.91$ & $\textbf{1.63}$ & $82.44$ & $1262.91$ & $2.38$ & $82.24$ & $1262.91$ & $2.04$ & $82.44$ \\
& Rcv1\_train.binary & $965.13$ & $\textbf{78.43}$ & $97.03$ & $965.13$ & $620.44$ & $97.00$ & $965.13$ & $3206.36$ & $97.15$ \\
& Real-sim & $13174.84$ & $\textbf{9.40}$ & $98.01$ & $13174.85$ & $96.91$ & $98.01$ & $13174.85$ & $785.63$ & $98.04$ \\
& Swarm (Aligned) & $115.10$ & $\textbf{147.98}$ & $96.46$ & $115.10$ & $199.90$ & $96.46$ & $115.10$ & $277.14$ & $96.46$ \\
& Swarm (Flocking) & $1027.82$ & $187.41$ & $94.13$ & $1027.82$ & $\textbf{130.88}$ & $94.13$ & $1027.82$ & $262.78$ & $94.09$ \\
& Swarm (Grouped) & $966.16$ & $\textbf{199.21}$ & $94.29$ & $966.16$ & $475.48$ & $94.25$ & $966.16$ & $281.29$ & $94.29$ \bigstrut[b] \\
\hline
\multirow{11}*{$0.05n$} & Arcene & $0.00$ & $0.90$ & $13.09$ & $0.00$ & $\textbf{0.63}$ & $24.94$ & $0.00$ & $1.23$ & $98.71$ \bigstrut[t] \\
& Dexter & $0.00$ & $0.19$ & $89.91$ & $0.00$ & $\textbf{0.15}$ & $90.81$ & $0.00$ & $2.08$ & $99.11$ \\
& Dorothea & $0.00$ & $0.59$ & $11.90$ & $0.00$ & $\textbf{0.32}$ & $11.90$ & $0.00$ & $4.82$ & $99.86$ \\
& Farm-ads-vect & $9.38$ & $*$ & $97.77$ & $9.38$ & $*$ & $97.69$ & $9.40$ & $*$ & $97.86$ \\
& Gisette & $0.00$ & $8.18$ & $8.24$ & $0.00$ & $\textbf{8.01}$ & $7.94$ & $0.00$ & $375.69$ & $78.42$ \\
& Madelon & $1225.19$ & $\textbf{2.25}$ & $72.26$ & $1225.19$ & $10.35$ & $71.86$ & $1225.19$ & $3.61$ & $72.26$ \\
& Rcv1\_train.binary & $470.53$ & $\textbf{275.16}$ & $96.06$ & $470.54$ & $2336.48$ & $96.04$ & $473.33$ & $*$ & $96.57$ \\
& Real-sim & $10087.92$ & $\textbf{27.83}$ & $96.37$ & $10087.93$ & $161.01$ & $96.36$ & $10087.93$ & $1771.75$ & $96.43$ \\
& Swarm (Aligned) & $6.80$ & $\textbf{296.74}$ & $95.34$ & $6.80$ & $429.97$ & $95.34$ & $6.80$ & $510.29$ & $95.34$ \\
& Swarm (Flocking) & $237.26$ & $\textbf{285.46}$ & $92.54$ & $237.26$ & $413.54$ & $92.54$ & $237.26$ & $532.85$ & $92.54$ \\
& Swarm (Grouped) & $210.80$ & $\textbf{260.47}$ & $92.71$ & $210.80$ & $937.63$ & $92.71$ & $210.80$ & $663.67$ & $92.71$ \bigstrut[b] \\
\hline
\end{longtable}
\end{landscape}

\begin{figure}[h]
\centering
\includegraphics[scale=0.465, trim = 2.5cm 12.75cm 1.9cm 0.8cm, clip]{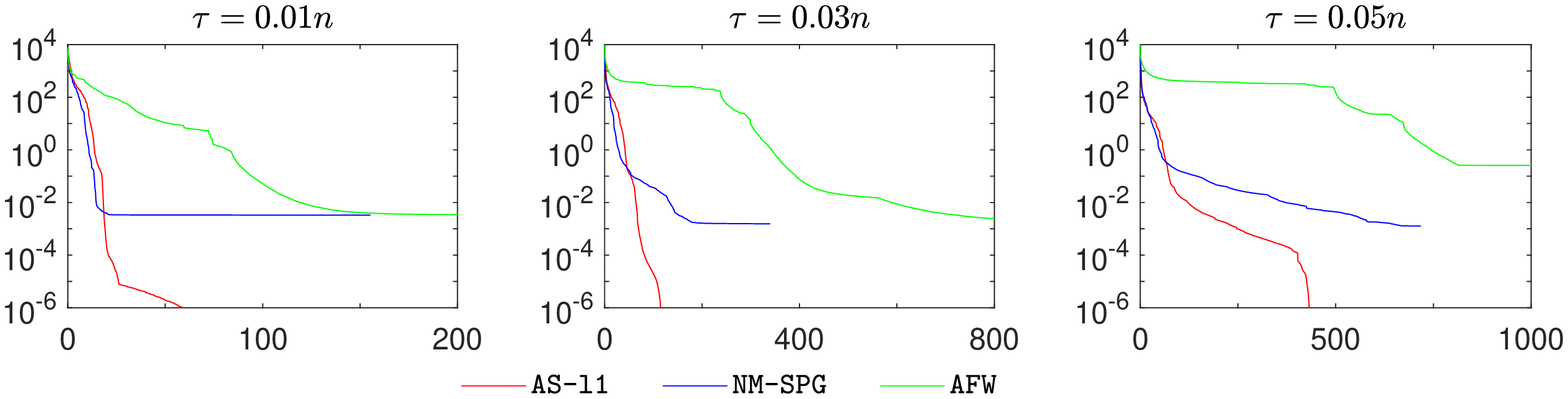}
\caption{Average optimization error \kblue{over} $\ell_1$-constrained logistic regression instances ($y$ axis) vs CPU time in seconds ($x$ axis). The $y$ axis is in logarithmic scale.}
\label{fig:plot_log_reg}
\end{figure}
In Figure~\ref{fig:plot_log_reg} we report the average optimization error $f(x^k)-f_{\text{best}}$ over the $11$ instances, for each value of $\tau$, with $f_{\text{best}}$ being the minimum objective value found by the algorithms.
We can notice that \texttt{AFW} is outperformed by the other two algorithms, which have similar performance when considering the average optimization error above $10^{-2}$.
When considering the average optimization error below $10^{-2}$, we see that \ASl1 outperforms \texttt{NM-SPG} too.

\section{Conclusions}\label{sec:conc}
In this paper, we focused on minimization problems over the $\ell_1$-ball and described
a tailored active-set algorithm.
We developed a strategy to guess, along the iterations of the algorithm, which variables should be zero at a solution. A reduction in terms of objective function value is guaranteed by simply fixing to zero those variables estimated to be active.
The active-set estimate is used in combination with a projected spectral gradient direction and a non-monotone Armijo line search. We analyzed in depth the global convergence of the proposed algorithm. 
The numerical results show the efficiency of the method on LASSO and sparse logistic regression instances, in comparison with two widely-used first-order methods.

\bibliographystyle{plain}
\bibliography{as_l1}   

\end{document}